\documentclass[12pt,leqno]{amsart}

\usepackage{amssymb,amsfonts,amsmath,amsthm}
\usepackage{latexsym}
\usepackage{verbatim}
\usepackage{epsfig}
\usepackage[active]{srcltx}
\usepackage[T1]{fontenc}
\usepackage[latin1]{inputenc}

\newtheorem{theorem}{Theorem}[section]

\newtheorem{corollary}[theorem]{Corollary}
\newtheorem{lemma}[theorem]{Lemma}
\newtheorem{proposition}[theorem]{Proposition}

\theoremstyle{definition}
\newtheorem{definition}[theorem]{Definition}
\theoremstyle{remark}
\newtheorem{remark}[theorem]{\sc Remark}
\newtheorem{example}[theorem]{\sc Example}

\newtheorem{question}[theorem]{\sc Question}
\newtheorem{note}[theorem]{\sc Note}

\renewcommand{\Box}{\square}    

\newcommand{\Aut}{{\rm{Aut}}}

\newcommand{\coker}{{\rm{coker\hspace{2pt}}}}

\newcommand{\Mon}{{\rm{Mon}}}

\newcommand{\im}{{\rm{Im\hspace{2pt}}}}

\newcommand{\cl}{{\rm{cl}}}

\newcommand{\ity}{{\infty}}

\newcommand{\m}{\setminus}
\newcommand{\fin}{\hspace*{\fill}$\Box$\vspace*{2mm}}

\newcommand{\bC}{{\mathbb C}}
\newcommand{\bP}{{\mathbb P}}

\newcommand{\bZ}{{\mathbb Z}}

\newcommand{\bQ}{{\mathbb Q}}

\title[Maps from a quasi-projective surface to a curve]{On the geometry of regular maps from a quasi-projective surface to a curve}
\author{A.J. Parameswaran}

\address{School of Mathematics, Tata Institute of Fundamental
Research, Homi Bhabha Road, Mumbai 400005, India}
\email{param@math.tifr.res.in}

\author{M. Tib\u ar}
\address{Math\'ematiques, UMR-CNRS 8524, Universit\'e Lille 1,
59655 Villeneuve d'Ascq, France.}
\email{tibar@math.univ-lille1.fr}



\date{March 30, 2015}

\begin{document}

\begin{abstract}

By exploring the consequences of the triviality of the monodromy group for a class of surfaces of which the mixed Hodge structure is pure, we extend results of Miyanishi and Sugie, Dimca, Zaidenberg  and Kaliman.
\end{abstract}
\maketitle
\section{Introduction}\label{intro}

In case of a polynomial function $P : \bC^2 \to \bC$,
 Miyanishi and Sugie \cite{MS} proved that the global monodromy group acting on $H^1$ of the general fibre is trivial if and only if the general fibre of $P$ is rational and $P$ is \textit{simple} (cf Definition \ref{d:simple}).
Dimca \cite[Theorem 1, (ii)]{Di} showed that in this statement the monodromy group can be replaced by the monodromy at infinity (i.e. around a very large circle in $\bC$). Kaliman \cite{Kal} proved that the number of reducible fibres of a primitive map is at most 
$\delta -1$, where $\delta = \#$horizontal components. 
Dimca also  observed that, in addition to the triviality of the monodromy, if all the fibres of the polynomial $P$ are irreducible, then  $P$ is linearisable, by using the above cited result of Kaliman and the celebrated theorem of Abhyakar-Moh-Suzuki. 
One has also studied the monodromy of a polynomial function on $\bC^2$ in relation to the monodromy of a plane curve germ, as well as extensions to several variables \cite{Di, ACD, NN1}.

We consider here instead of $\bC^2$ some other classes of surfaces verifying certain Hodge theoretic conditions in order to draw topological properties of the monodromy. This is also motivated by long standing research on classes of affine surfaces. In several papers \cite{GP}, \cite{GPS}, \cite{GS1}, \cite{GS2}
Gurjar, Pradeep  and Shastri proved that every
$\bQ$-homology plane is rational, question which was raised by Miyanishi. This is already a big class of varieties, since it has moduli.

Let $X$ be a nonsingular, quasi-projective, connected surface,  and let $f: X \to C$ be a regular map onto a nonsingular affine curve.   Let $\bar f : \bar X \to \bar C$ be some  resolution of the indeterminacy points at infinity, namely $\bar X$ and $\bar C$ are nonsingular projective spaces and  such
that $D= \bar X \m X$ is a simple normal crossing divisor, i.e.,  $D = \cup_i D_i$, where the irreducible components $D_i$ are smooth divisors which intersect transversely and not more than two pass through the same point. We shall refer to $\bar f$ as ``compact resolution'' and we shall work throughout the paper with the condition that $D$ is connected.

 We study here the monodromy of the regular map $f$ in relation to that of $\bar f$. 
The restriction of $\bar f$ to some component $D_i$ is either constant and then one says that $D_i$ is \textit{vertical}, or $\bar f_{|} : D_i \to \bar C$ is onto,  it is namely a (ramified) covering of degree $d_i$ and we 
say that $D_i$ is a \textit{horizontal}  component. We denote by $F_p$ and $\bar F_p$ the fibres of $f$ and of $\bar f$ over some point $p\in C$ or $p\in \bar C$, respectively. Also, $F$ and $\bar F$ without lower index will stand for general fibres of $f$ and $\bar f$, respectively. One has to take into account that $\bar F$ is indeed the closure of the general fibre $F$, whereas $\bar F_p$ is not that of $F_p$, since it may contain vertical components.

The aim of this paper is to study maps from surfaces $X$ subject to a specific Hodge theoretical property which we call ``$H^2$-pure and $H^3$-pure surfaces'', meaning that the mixed Hodge structure of $H^i(X)$ is pure of weight $i$, see Definition \ref{d:pure}, Remark \ref{r:pure} and Example \ref{r:pure}.

As explained in \S \ref{pure}, this property 
encompasses a whole bunch of topological conditions  and controls the cohomology of the fibres of $f$ along with the triviality of the monodromy. Let us point out that the $H^i$-purity is an intrinsic Hodge theoretic property of $X$ whereas its topological interpretations given in \S \ref{pure} 
are in terms of the divisor $D=\bar X \m X$ and of the map $f : X \to C$.

We show that we can treat with the same condition ``$H^2$ and $H^3$-pure surface'' not only Miyanishi-Sugie's and Dimca's type questions but also questions raised by Kaliman's and Zaidenberg's results \cite{Kal}, \cite{Za1}.

The paper is organised as follows.
 In \S \ref{s:monodr} we prove that if the monodromy group $\Mon^1 \bar f$ acting on the cohomology group $H^1(\bar F)$ of the general fibre is trivial then the restriction morphism $H^1(\bar X) \to H^1(\bar F_p)$ is a surjection for all $p\in \bar C$. We moreover show that
the kernel of this morphism is precisely $H^1(\bar C)$, for any choice of $p\in \bar C$, whenever the general fibre is connected. 
The proof uses Deligne's invariant cycle theorem and several mixed Hodge theoretic facts. 
We then relate the triviality of the monodromy group of $f$ to that of $\bar f$ and to the simplicity of $f$ (Propositions \ref{p:trivmon} and \ref{p:equiv}).

In \S \ref{pure}, we derive topological consequences of the purity of the cohomology (Lemma \ref{l:h2},  Proposition \ref{p:pure}, Corollary \ref{c:tree}).
We analyse the horizontal and vertical components of a map from a $H^2$- and $H^3$-pure surface $X$ onto an affine curve $C$ (Theorem \ref{c:points}) and combine with Deligne's invariant cycle theorem (also used by Dimca in \cite{Di}) in order to show the determination of the monodromy group by a single loop around some point $c\in \bar C \m C$ (Theorems \ref{t:genus}, \ref{t:monpoint}), thus extending  results by Miyanishi-Sugie \cite{MS} and Dimca \cite{Di}.

In \S \ref{s:bound} we prove that the number of reducible fibres of $f$ is bounded from above by $\delta -1 + b_2(X)$, where $\delta$ is the number of horizontal components of $D$ (Theorem ~\ref{t:one}). This takes into account the non-trivial cohomology of $X$ and extends  Kaliman's result  for $X=\bC^2$ \cite{Kal} where this bound is $\delta -1$.  In a different direction there are very interesting bounds of the \emph{total order of reducibility} of a polynomial  
by Y. Stein \cite{St},  of pencils of plane projective curves by Lorenzini \cite{Lo}, which are  altogether generalised (including Kaliman's result \cite{Kal}) to pencils of hypersurfaces by Vistoli \cite{Vi}.

We finally show in \S \ref{localtriviality} that if the monodromy group is trivial and all the fibres are irreducible and reduced, then $f$ is a locally trivial fibration (Theorem \ref{t:fibration}(b)), thus extending Dimca's result \cite{Di}. 

Moreover, if in the above statement we replace ``all fibres are  reduced'' by ``at least one fibre is non-reduced '' then we can show that the general fibre of $f$ is either $\bC$ or $\bC^*$ (Theorem \ref{t:fibration}(c)).  This is closely
related to Zaidenberg's result \cite[Lemma 3.2(b)]{Za1} since the conclusions are the same, whereas the hypotheses are totally different: Zaidenberg has a local assumption that $f$ is affine or Stein but no other assumption on the local geometry of the divisor at infinity, while we assume the
$H^2$- and $H^3$-purity of $X$. On the other hand, \cite{Za1} assumes the constancy of the Euler characteristic while we assume here the  triviality of the monodromy group. 

\smallskip
\noindent
\emph{Acknowledgement.} The authors express their gratitude to the Mathematische Institut Oberwolfach for hosting them in the RiP program. They thank the School of Mathematics at TIFR Mumbai, and the Paul Painlevé Laboratory at Université Lille 1 for supporting their visits during which this project was accomplished.


\section{Open and projective monodromies}\label{s:monodr}

It is well known that there is a finite set of points $A\subset \bar C$ such that, for any $c\in \bar C \m A$ the restriction:
\[  \bar f_| : \bar X \cap \bar f^{-1}(\bar C \m A) \to \bar C \m A \]
is a proper stratified submersion hence a locally trivial fibration. In particular the map $f_| :  X \cap f^{-1}(C \m A) \to  C \m A$ is a locally trivial fibration. 

Let $\bar F$ and $F$ denote the general fibres of $\bar f_|$ and $f_|$ respectively. The points $\bar F\m F$ are precisely the intersections of $\bar F$ with the horizontal curves. For any point $p\in \bar F\m F$ one defines a generator of the group $H^1(F) \simeq H_1(F)$ as the equivalence class of the oriented boundary of a small enough disk at the point. It is well-known that there is a single relation among all these $\delta := |\bar F\m F|$ generators.

We consider homology and cohomology with coefficients in $\bC$. The monodromy groups $\Mon^k \bar f$ and $\Mon^k f$, for $k=0, 1, 2$, are by definition the monodromy representations of the fundamental groups,  $\pi_1(\bar C \m A) \to \Aut(H^k(\bar F))$ and $\pi_1(\bar C \m A) \to \Aut(H^k(F))$, respectively. We are more specifically interested in the first cohomology group and its global invariant cycles, denoted by $H^1(\bar F)^\pi$ and $H^1(F)^\pi$, respectively.

\begin{definition}
We say that $f$ is primitive if its general fibre $F$ is connected.
\end{definition}
Note that $f$ is primitive if and only if some compact resolution $\bar f$ of $f$ is primitive. Actually this does not depend of the choice of compact resolution (as defined in \S \ref{intro}).

By the Stein factorisation, there exists a projective curve $\bar C'$ and regular maps $\bar f' : \bar X \to \bar C'$ and $\mu : \bar C' \to \bar C$ such that $\bar f = \mu \circ \bar f'$ and that $\bar f'$ is primitive.

\begin{proposition}\label{p:mono}
 Let $p\in \bar C$ and let $\bar f$ be primitive. Then the following sequence is exact:
\begin{equation} \label{eq:deligne}
 0 \to H^1(\bar {C}) \to H^1(\bar {X})\to H^1(\bar {F_p}).
\end{equation}
 Moreover, $\Mon^ 1 \bar f = 1$ if and only if 
 the last map is a surjection and $H^1(\bar F_p)\simeq H^1(\bar F_q)$  for all $q \in \bar C$. 
\end{proposition}

For the proof, we need two lemmas.
\begin{lemma}\label{l:leray}
Let $f:E \to B$ be a primitive projective submersion, where $B$ is non-complete. 
Consider the Leray spectral sequence with $E_2^{pq}= H^p(B, R^q f_* (\bC_E))$ abutting to $H^{p+q}(E)$. 
If  $B$ and $F$ are complex curves, then we have an exact sequence:
 $$ 0 \to H^1(B) \to H^1(E)\to H^1(F) $$ 
\end{lemma}
\begin{proof} The Leray spectral sequence degenerates  in $E_2$ by \cite[Theorem 4.15]{Vo}.
The only possibly nontrivial differential in the spectral sequence is $H^0(B, 
R^1 f_* (\bC_E))\to H^2(B, H^0(F))$. Since $H^2(B)=0$ for a noncomplete curve, this differential must be $0$.
The 
quotient of $H^1(E)$ by  $H^1(B, R^0 f_* (\bC_E))$  is $H^0(B, R^1 f_* (\bC_E))$, which is precisely the invariant cycles. Note that  $H^1(B, R^0 f_* (\bC_E))$ is isomorphic to $H^1(B)$ since $R^0 f_* (\bC_E) = \bC_B$ by the primitivity.
On the other hand, the invariant cycle theorem 
of Deligne states that the invariant cycles are precisely the image of 
$H^1(E)\to H^1(F)$. Thus our sequence is exact.
\end{proof}
We shall use the following Hodge theory fact, see e.g. \cite[4.3.3]{Vo} for a proof:

\begin{lemma} \label{l:open}
 Let $Y$ be a connected nonsingular projective variety and let $j:U\hookrightarrow Y$ denote the inclusion of some Zariski open subset $U$.  
Then, for any $k\ge 0$, the weight $k$ filtration of the mixed Hodge structure 
on  $H^k(U)$ is precisely the image of the morphism induced by the restriction $j^*: H^k(Y) \to H^k(U)$.
\fin
\end{lemma}

\subsection*{Proof of Proposition \ref{p:mono}}
 We first prove the second statement. If $\bar {F}$ denotes
the general fibre of $\bar f$ then, by Deligne's invariant cycle theorem \cite{De}, the global invariant cycles coincide with 
the image of $H^1(\bar {X}) \to H^1(\bar {F})$. Then $\Mon^ 1 \bar f = 1$ if and only if this map is surjective. For proving the surjectivity for a special fibre $\bar {F_p}$ we use the following construction. Let $U$ be a tubular neighbourhood of $\bar {F_p}$ which is
invariant under $\overline {f}$. Let $\bar {F}\subset U$ be a
general fibre contained in $U$. Then $H^1(U)\cong H^1(\bar {F_p})$
is an isomorphism and  $H^1(\bar {X})\to H^1(\bar {F})$ factors
through  $H^1(U)$. Hence $H^1(U)\cong H^1(\bar {F_p})\to
H^1(\bar {F})$ is surjective. Since $b_1(\bar F) \geq b_1(\bar F_p)$,
this must be an isomorphism and the proof of the surjection in (\ref{eq:deligne}) is done. For the sake of completeness, let us provide here a brief proof of the inequality $b_1(\bar F) \geq b_1(\bar F_p)$ for any special fibre $\bar F_p$ of $\bar f$. We consider the pair $(X_{\Delta^ *}, \bar F)$ where $\Delta$ is a small disk centered at the image point $\bar f (\bar F_p)$ and $\Delta^ *$ is the punctured disk. We have $H_*(X_{\Delta^ *}, \bar F) \simeq \tilde H_* ( S^ 1 \vee S(\bar F))$ where $S(\bar F)$ denotes the suspension over the general fibre  $\bar F$. It follows that $H_1(X_{\Delta^ *}, \bar F) \simeq \bC$. The exact sequence of the pair yields:
\begin{equation}
 \begin{array}{cccccc}
 H_1(\bar F) & \stackrel{i}{\rightarrow} &  H_1(X_{\Delta^ *}) & \stackrel{k}{\rightarrow} & \bC & \rightarrow 0 \\
j \downarrow & \ & \downarrow \pi & \ & \ &  \\
H_1(\bar F_p) & \simeq & H_1(X_{\Delta}) & \ & \ &
\end{array}
\end{equation}
 where $j$ is induced by the restriction to $\bar F$ of the topological contraction of a tube to the special fibre $\bar F_p$. The surjectivity of the natural map $\pi$ holds already at the level of fundamental groups since $X_{\Delta^ *}$ is the complement of a proper complex analytic subset. Since there exists a lift by $k$ of $[S^1]$ which is sent to $0$ by $\pi$, it follows that the restriction $\pi_{| \ker k}$ is a surjection too. By the exactness, $\pi \circ i$ is surjective and therefore $j$ is surjective too. Hence $b_1(\bar F) \geq b_1(\bar F_p)$.
\smallskip

To prove the exactness of the sequence (\ref{eq:deligne}), let us show that the kernel of $\gamma : H^1(\bar {X})\to H^1(\bar {F})$
is $H^1(\bar {C})$.  First we observe that $\ker \gamma$ has weight 1 since $H^1(\bar X)$ has pure Hodge structure. Let $B\subset \bar {C}$ be the open subset over which $\bar {f}$ is submersive. Let $\bar {X}_B:=\bar {f}^{-1}(B)$.
We have the following commutative diagram, where the botom sequence is exact by Lemma \ref{l:leray}:
\begin{equation}
 \begin{array}{cccccccc}
 0 & \to & H^1(\bar B) & \rightarrow &  H^1(\bar X) & \stackrel{\gamma}{\rightarrow} & H^1(\bar F) \\  \ & \ & 
 \downarrow & \ & \downarrow & \ & \downarrow =  \\
0 & \to & H^1(B) & \rightarrow & H^1(\bar X_B) & \rightarrow & H^1(\bar F) 
\end{array}
\end{equation}

 The restriction maps $H^1(\bar B) \to H^1( B)$ and $H^1(\bar X) \to H^1(\bar {X}_B)$ are injective since the inclusion of a Zariski open subset in a smooth variety induces a surjection of their fundamental groups.
 This implies that $H^1(\bar B) \to H^1(\bar {X})$ is injective and that we have the equality 
 $\ker \gamma = H^1(\bar {X}) \cap H^1(B)$, 
where the intersection is taken in $H^1(\bar {X}_B)$, i.e. $H^1(B)$ and $H^1(\bar {X})$ are identified with their isomorphic images in $H^1(\bar {X}_B)$. 
Since by Lemma \ref{l:open} the weight 1 part of $H^1(B)$ is precisely the image of $H^1(\bar {B})$ and since $\ker \gamma$ is pure of weight 1, we conclude that $H^1(\bar {X}) \cap H^1(B)= H^1(\bar B)$. This ends our proof since $\bar B = \bar C$ by definition.

The following useful result is a simple consequence of the Stein factorisation.
\begin{lemma}\label{l:connected}
Let $\bar X \to \bar C$ be a regular map between smooth projective varieties, where $\bar X$ is irreducible and $\bar C$ is a curve. Let $D_0\subset \bar X$ be a horizontal curve. Then, for every $p \in \bar C$, $D_0$ meets every connected component of the fibre $\bar F_p$.
\fin
\end{lemma}

\begin{definition}\label{d:simple}
 We say that $f$ is \textit{simple} if $\deg \bar f_{|D_i} = 1$ for all horizontal $D_i$. This does not depend on the choice of the resolution $\bar f$.
\end{definition}

We draw some consequences of the triviality of the monodromy groups $\Mon^k f$ and $\Mon^k \bar f$, respectively. In the particular case of $f: \bC^2 \to \bC$, a version of the following Propositions \ref{p:trivmon} and \ref{p:equiv} has been proved  in \cite{Di}, \cite{NN1}.
The notation $\Mon$ without upper index means all $k=0, 1, 2$.
\begin{proposition} \label{p:trivmon}
Let  $\Mon f= 1$. Then:
\begin{enumerate}
 \item  $f$ is primitive.
 \item  $\Mon \bar f= 1$ 
 \item $f$ is simple.
\end{enumerate} 

\end{proposition}
\begin{proof}

(a). If $f$ was not primitive then the \textit{Stein factorisation}  allows one to write $f = g\circ h$, where $h : X \to C'$ and $g : C' \to C$ is a  covering of  degree $d>1$. 
Since $C'$ is connected, $\Mon^0 f$ identifies to the monodromy group of the covering $g$. The later is a subgroup of the permutation group of $d$ points. This permutation subgroup is not trivial if $d>1$. This actually shows that (a) is equivalent to $\Mon^0 f= 1$.

\smallskip

\noindent (b). We have the short exact sequence induced by the inclusion $F\stackrel{i}{\to} \bar F$:
\begin{equation} \label{eq:exact1}
 0  \to H^1(\bar F) \stackrel{i^*}{\to} H^1(F) \to \coker i^* \to 0
\end{equation}

\noindent
where $\coker i^* \simeq \tilde H_0(\bar F\m F)$. The later is of rank $\delta -1$, where $\delta$ is the number of points at infinity $\bar F\m F$. Moreover the geometric monodromy of $\bar f$ described in \S \ref{intro} acts on $F$. This sequence shows that the triviality of the monodromy group of $f$ implies the triviality of the monodromy group of $\bar f$.  

\smallskip

\noindent (c). Let us assume that there is at least one horizontal component, otherwise we have nothing to prove.
Let $\{ p_1, \ldots , p_\delta \} = D \cap \bar F$ and let $\Delta_j$ denote some small closed disk around the point $p_j$. The complex normal bundle of some horizontal divisor $D_0$ within $X$ has a well defined orientation given by the complex orientation of $X$ and that of $D_0$. The general fibres of $\bar f$ are transversal to $D_0$ and their orientation coincides with the one of the normal bundle. This implies in particular that all the small circles $\partial \Delta_j$ have a canonical orientation. 

Let us consider the homology exact sequence of the pair of general fibres $(\bar F, F)$:
\[ 0\to H_2(\bar F) \to H_2(\bar F, F) \stackrel{\nu}{\to} H_1(F) \to H_1(\bar F) \to 0
\]
where one may identify $H_2(\bar F, F)$ with $\oplus_{j=1}^\delta H_1(\partial \Delta_j)$ by excision and boundary isomorphism, and where $H_2(\bar F) = \bZ$.
 Denote by $e_j := \nu([\Delta_j, \partial \Delta_j])$ the generators of $\im \nu \subset H_1(F)$. 
We may assume that all the points $p_j$ are in some disc $\Delta \subset \bar F$ containing also the small disks $\Delta_j$. Then the unique relation among the $e_j$'s in $H_1(F)$ is precisely $e_1 + \cdots + e_\delta = \nu([\Delta, \partial \Delta]) =0$.

If $f$ is not simple then there exists some horizontal component $D_0$ of the divisor $D = \bar X \m X$ such that $\deg (\bar f_{|D_0})\ge 2$.
The geometric monodromy along any loop in $\bar C$ which avoids bifurcation values acts on the generators $e_j$ corresponding to the points $\bar F \cap D_0$
by precisely an orientation preserving permutation of these $e_j$'s. Say we have $e_j \mapsto e_k$ for some $k\not= j$. But $e_j = e_k$ cannot be a relation since this would be different from the unique one presented above. It therefore follows that the monodromy is not the identity, which is  a contradiction.
\end{proof}

We have the following partial reciprocal of Proposition \ref{p:trivmon}.

\begin{proposition} \label{p:equiv}
  Let $f$ be simple. If $\Mon \bar f =1$ then $\Mon f$ is unipotent. 
\end{proposition}

\begin{proof}  
If $\Mon^0 \bar f = 1$ then $\bar F$ is connected (since $X$ is smooth and irreducible), hence $F$ is connected, so $\Mon^0 f = 1$. We also have that the general fibre $\bar F$ is smooth and irreducible. Then either $F$ is an open set and then $H^2(F) = 0$, or $F= \bar F$ and then $\Mon f = \Mon \bar f =1$. 
In the exact sequence \eqref{eq:exact1}
 the monodromy acts as the identity on both sides. This proves the unipotency claim for $\Mon^1 f$.
\end{proof}

\begin{question}\label{c:conj}
Can one find an example of $f$  simple and with $\Mon \bar f =1$ but such that  $\Mon f \not= 1$?
\end{question}


\section{Topology of regular maps from $H^i$-pure surfaces}\label{pure}


\begin{definition}\label{d:pure}
A quasi-projective surface $X$ is called {\em $H^i$-pure} for $i \ge 1$ if the mixed Hodge structure on the $i$-th cohomology $H^i(X)$ is  pure of weight $i$. 
\end{definition}

 These conditions are extending the case $X= \bC^2$. For instance $\bQ$-homology planes are $H^i$-pure for any $i\ge 1$ since they have trivial cohomology. 
 
 Throughout this paper we shall impose the following specific condition: ``$X$ is  $H^2$-pure and $H^3$-pure''. 
 
\begin{remark}\label{r:pure}
 Let us consider the particular case $X\subset \bP^2$. Then we have the following characterization:
 
 $X$ is $H^2$-pure and $H^3$-pure if and only if $Y:= \bP^2 \m X$ is either a divisor of which the dual graph is a connected tree of topological $\bP^1$'s, or $Y$ is one point.
 
 Actually the $H^3$-purity is the easy part and just means the connectivity of $Y$, whereas the $H^2$-purity means all the rest, and thus it is more subtle.
 
 In particular  $X:= \bC \times \bC^{*k}$ (where $\bC^{*k}$ denotes the complex line with $k$ punctures) verifies the hypotheses for any $k \ge 0$, which already provides an infinite family of surfaces in $\bP^2$.
\end{remark}

Let us now interpret the $H^i$-purity in terms of the normal crossing divisor $D$ obtained after resolution.

\begin{lemma} \label{l:h2}
 Let X be a smooth quasi projective surface. Let $\bar X$ be a connected smooth
compactification of $X$ with $D := \bar X \m X$  a divisor with simple normal crossing. Then:
\begin{enumerate}
 \item   $X$ is $H^3$-pure if and only if the restriction map $H^0(\bar X) \to H^0(D)$ is surjective, i.e., if and only if $D$ is connected.
\item   $X$ is $H^2$-pure if and only if $H^1(\bar X) \to H^1(D)$ is surjective.
 \item  $X$ is $H^1$-pure if and only if $H^2(\bar X) \to H^2(D)$ is surjective, i.e., if and only if the components of $D$  are linearly 
 independent in $H_2(\bar X)$.
\end{enumerate}
\end{lemma}

\begin{proof}
Consider the following long exact sequence of the pair $(\bar X,D)$
$$   \cdots \to H^i(\bar X) \to  H^i(D) \to H^{i+1}(\bar X,D)  \to H^{i+1}(\bar X) \to \cdots
$$
for $i=0, 1, 2$.
We have the following  sequence of equivalences: $H^i(\bar X)  \to  H^i(D)$  is surjective $\Leftrightarrow$  $H^{i+1}(\bar X, D) \to  H^{i+1}(\bar X)$ is injective $\Leftrightarrow$  the Lefschetz dual map induced
by inclusion  $H_{4-i-1}( X) \to  H_{4-i-1}(\bar X)$ is injective $\Leftrightarrow$ $H^{4-i-1}(\bar X) \to  H^{4-i-1}(X)$ is surjective  (by the universal coefficient theorem on cohomology with complex coefficients)
 $\Leftrightarrow$ $X$ is $H^{4-i-1}$-pure, by Lemma \ref{l:open}.

To finish the proof of (c), let us note that the surjectivity of  $H^2(\bar X)\to H^2(D)$  is equivalent to the injectivity of the dual map $H_2(D)\to H_2(\bar X)$ induced by the Universal Coefficient Theorem over a field, which just means the independence 
of the irreducible components of $D$ in the 2-homology  of $\bar X$.
\end{proof}

\begin{example}\label{e:pure}
If  $\bar X$ is a smooth projective surface and $D$ a connected divisor whose components form a tree of rational curves, then $X := \bar X \m D$ is a $H^2$- and $H^3$-pure surface.
\end{example}

The above result and its proof can  easily be extended to any dimension $n$, as follows:
\begin{proposition}\label{p:pure}
\textit{Assume $X$ is a smooth connected complex variety and $\bar X$  is a completion
to a smooth projective (complete) variety with the complement  $D:= \bar X \m  X$.
 Then
 $X$ is $H^{2n-i-1}$-pure if and only if $H^i(\bar X)  \to  H^i(D)$ is surjective.}
 \fin
\end{proposition}

\begin{corollary}\label{c:tree}
If $X$ is a $H^2$-pure surface, then  $H^1(D)$ has a pure Hodge structure and hence  the dual graph of $D$ is a tree.
\end{corollary}

\begin{proof}
Since $\bar X$ is a smooth complete variety, $H^1(\bar X)$ has a pure Hodge structure. We have shown above that $H^1(\bar X)\to H^1(D)$ is surjective, hence $H^1(D)$  has a pure structure too.
By the classical theory of weights (see e.g. \cite{Vo}) it follows that the weight $0$ part of $H^1$ of a normal crossing curve 
arises from the loops in the dual graph of $D$. Hence a simple normal crossing curve has $H^1$ pure if and 
only if its dual graph is a tree. 
\end{proof}

\begin{theorem}\label{p:points}\label{c:points}
 Let  $X$ be a $H^2$-pure and $H^3$-pure surface and let $f:X \to C$ be a 
map onto an affine curve.  Let $\bar f : \bar X \to \bar C$ be some compact resolution of $f$.
\begin{enumerate}
 \item If $f$ is proper then there is no horizontal component, $\bar C \m C = \{ p \}$ and  $D= \bar F_p$.
 
\item If  $b_3(X) > 0$, then  $f$ is proper.

\item 
If there is more than one horizontal component, then all horizontal components are rational,
     and $C\simeq \bC$.
 
\item  If $\bar f^{-1}(\bar C \m C)$ is not connected, then there is a unique horizontal component, and $\bar f$ has no invariant cycles.    
  \end{enumerate}
\end{theorem}   

\begin{proof}
\noindent
(a). We have the equivalence $f$ is proper $\Leftrightarrow$ there is no horizontal component. From the $H^3$-purity we get that $D$ is connected and therefore $\bar C \m C = \{ p \}$ and  $D= \bar F_p$. Note that we have not used the $H^2$-purity here.

\noindent
(b).
 The $H^2$-purity and  $H^3$-purity yield the exactness of the following sequence
  $$0 \to H^1(\bar X, D) \to H^1(\bar X) \to H^1(D) \to 0$$
 which implies the equality  $h^1(\bar X) = h^3(X) + h^1(D)$, 
where $h^i$ denotes the dimension\footnote{within this proof we prefer the notation $h^i$ for the betti numbers instead of the notation $b_i$ used everywhere else in the paper.} of the corresponding cohomology $H^i$.

Let us first assume that $f$ is primitive.
Let's then fix some $p\in \bar C\m C$. From Proposition \ref{p:mono} we have an injection $H^1(\bar X)/H^1(\bar C) \to H^1(\bar F_p)$.   It then follows:

\begin{equation}\label{eq:inequality}
h^1(\bar F_p)  \ge  h^1(\bar X)   -  h^1(\bar C) = h^3(X) + h^1(D) - h^1(\bar C).
\end{equation}

Let us write additively $D= \bar F_p + \sum_{q\in \bar C\m (C\cup \{ p\})} \bar F_q + D_0 + \sum_{i\neq 0} D_i + D_v$ as a divisor, without counting the multiplicities,
where $D_i$ for $i\ge 0$ denote the horizontal components, $D_v$ is the sum of the vertical components over points
of $C$,  $\bar F_p$ and $\bar F_q$ are the fibres of $\bar f$ over   $p, q \in \bar C \m C$. 
By substituting in \eqref{eq:inequality}, we obtain:

\sloppy 
$$h^1(\bar F_p) \ge 
                       h^3 (X)  + [ h^1(D_0) - h^1(\bar C) ]    +  h^1(\bar F_p)
                      + \sum_{q\neq p} h^1(\bar F_q)  +  \sum_{i\neq 0} h^1(D_i)  + h^1(D_v).$$
                      
Hence by cancelling $h^1(\bar F_p)$ we get:

\begin{equation}\label{eq:equality0}
 0  \ge h^3(X) + [ h^1(D_0) - h^1(\bar C) ]  + \sum_{i\neq 0} h^1(D_i) + h^1(D_v) + \sum_{q\neq p} h^1(\bar F_q).
\end{equation}

By Riemann-Hurwitz, the term $h^1(D_0) - h^1(\bar C)$ of \eqref{eq:equality0}  is non-negative whenever $D_0$ is not empty, i.e., $D$ has at least one horizontal component, equivalently,  $f:X\to   C$ is not proper. Since the other terms are obviously non-negative, the above inequality implies that all the terms are 0. Then $h^3(X) =0$, which is a contradiction.  We also must have $h^1(\bar C)>0$ since this term is the only one with negative sign in the inequality \eqref{eq:equality0}.

Now if $f$ is not primitive, then we apply Stein factorisation and deduce that this is proper, hence $f$ is proper by the commutativity of the factorisation diagram.

\smallskip
\noindent
(c).  First apply the Stein factorisation and use the inequality \eqref{eq:equality0} for $\bar f'$.   Let $D_1 \not= D_0$ be some other horizontal component. Then $h^1(D_1) =0$, hence $h^1(\bar C') =0$, so $h^1(D_0) =0$. Moreover, since $D$ is a tree, we get that there is no point $q\not= p$, hence $| \bar C' \m C'| =1$, and thus $C \simeq \bC$. Indeed, by Lemma \ref{l:connected}, $D_0$ and $D_1$ would produce a cycle in the dual graph of $D$ together with some connected component of $\bar F_p$ and some connected component of $\bar F_q$, and this contradicts Corollary \ref{c:tree}.

\smallskip
\noindent
(d). 
Two horizontal components  would produce a loop in the dual graph of $D$ together with two of the connected components of $\bar f^{-1}(\bar C\m C)$, which contradicts the fact that $D$ is a tree, by Lemma \ref{l:connected}.
 
  After taking the Stein factorisation $\bar f'$,  the inequality \eqref{eq:equality0} shows that the fibre $\bar F'_p$ of $\bar f'$ may
have nontrivial $H^1$ for at most one point $p\in \bar C' \m C'$. But since $|\bar C' \m C'|\geq 2$ 
 it follows from \eqref{eq:equality0} that $H^1(\bar F'_p)=0$ for all points $p\in \bar C'\m C'$.  By Proposition \ref{p:mono}, $H^1(\bar X) \simeq H^1(\bar C')$ and the map $H^1(\bar X) \to H^1(\bar F)$ is zero over $\bar C'$, hence over $\bar C$ too.   Therefore $\bar f$ has no global invariant cycles, by Deligne's theorem \cite{De}.
 
\end{proof}

  \begin{theorem}\label{t:genus}
 Let  $X$ be a $H^2$-pure and $H^3$-pure surface and let $f:X \to C$ be a non proper 
map onto an affine curve. Let $\bar f : \bar X \to \bar C$ be some compact resolution of $f$. 

If $g(\bar{C}) > 0$ or if $| \bar C \m C| \ge 2$  then $f$ has a unique horizontal component.

  If moreover $f$ is primitive and either $g(\bar{C}) > 0$ or $| \bar C \m C| > 2$, then $f$ is also simple.
\end{theorem}

\begin{proof}
  If $|\bar C \m C| \ge 2$ then  $\bar f^{-1}(\bar C \m C)$ is not connected and one applies  Theorem \ref{c:points}(d) to get that there is a unique horizontal component.

Let us assume now $g(\bar{C}) > 0$.
We consider the Stein factorisation $\bar f' : \bar X \to \bar C'$.  From  (\ref{eq:equality0}) and by the inequality $g(D_0) \ge g(\bar C') \ge g(\bar C) >0$ which is a consequence of the  Riemann-Hurwitz formula, it follows  from Theorem \ref{c:points}(c)
that there is only one horizontal component $D_0$ for $\bar f'$, hence a unique horizontal component for $\bar f$.

For our second claim, let us now assume that $f$ is primitive. 

If the degree $d$ of  $D_0 \to \bar C$ is $>1$ then either $g(\bar C) =1$ or $g(\bar C) =0$. If $g(\bar C) =1$
then there are no ramifications. But on the other hand,
since $D$ is a tree (Corollary \ref{c:tree}), the horizontal component must be totally ramified over any point in $\bar{C} \setminus C$ since otherwise it will produce cycles in the dual graph of $D$. This gives a contradiction to $d>1$. Thus the degree $d$ must be 1, hence $f$ is simple. 
 
If $g(\bar C) =0$ then, by Riemann-Hurwitz, either $d=1$ hence $f$ is simple, or $d>1$ and  there must be at least a ramification. In the later case, since the dual graph of $D$ is a connected tree, all such ramifications over some point of $\bar{C} \setminus C$ must be total and by Riemann-Hurwitz we get  $|\bar C \m C| \le 2$, which contradicts our assumption $| \bar C \m C| > 2$.
\end{proof}

\begin{example}\label{ex:simple}
This shows that if $f$ is not primitive then it might be not simple. Let $C_1$ and $C_2$ be two smooth projective curves of  genus $> 1$. Let $h : C_1 \to \bP^1$ be some regular map such that $\deg h >1$ and $h^{-1}(\infty) = p$, for some $p\in C_1$. Let $\bar X := C_1 \times C_2$ and $D:= C_1 \times \{ q\} \cup \{ p\}\times C_2$ be the union of  sections of  different projections. Since the dual graph of $D$ is a connected tree, $X := \bar X\m D$ is $H^2$-pure and $H^3$-pure. Actually this is even $H^1$-pure, see Lemma \ref{l:h2}(c).

Let $\pi : C_1 \times C_2 \to C_1$ denote the projection and let $f := h\circ \pi :  X \to \bC$. Then $f$ is not primitive and not simple. One may produce other examples taking instead of $h$ some map $C_1 \to C_3$ with $g(C_3)>0$ which is totally ramified at $p$.
\end{example}

The next result extends Dimca's \cite[Theorem 1(i-iii)]{Di} for polynomials of two variables. 

\begin{theorem}\label{t:monpoint}
Let $X$ be an $H^2$ and $H^3$-pure surface and $f:X\to C$ be a regular map onto an affine curve $C$ which is not proper.
Let $p\in \bar{C}\m C$ and $\gamma_p$ be a `small' loop around $p$. Let $T_p : H^1(F) \to H^1(F)$ be the  monodromy of $f$ defined by $\gamma_p$ and assume that $\bar F_p$ is connected. 

Then:

\begin{enumerate}
\item If $T_p =1$ then $f$ is simple and $\Mon \bar f=1$. If $f$ is simple and $\Mon \bar f=1$ then $T_p$ is unipotent.

 \item If $\bar F_p$ is a tree of rational curves then the eigenvalue 1 of the matrix of $T_p$ occurs in  Jordan blocks of size 1 and its multiplicity is the number of horizontal divisors minus one.
 \item  If the general fibre $\bar F$ is rational then $T_p$ is finite.
 \end{enumerate}

\end{theorem}

\begin{proof}  (a). The fact that $D$ is a connected tree and that $\bar F_p$ is connected implies that every horizontal component of $D$ is totally ramified over the point $p$. The triviality of $T_p$ 
implies in addition that the local degree of $\bar f_{|D_i}$ over $p$ must be $1$, for any  horizontal component $D_i$, i.e., $f$ must be locally simple over $p$ and therefore simple.  

If $T_p=1$ then $\bar T_p : H^1(\bar F) \to  H^1(\bar F)$ is the identity. Then the \textit{local invariant cycle theorem} \cite{Cl} tells that the image of the map induced by the restriction from a small tubular neighbourhood $H^1(U_p)\cong H^1(\bar F_p) \to H^1(\bar F)$ is surjective, 
 where $\bar F$ denotes here the general fibre in $U_p$. Therefore one has an isomorphism $H^1(\bar F_p) \cong H^1(\bar F)$.
Since $\bar F_p\subset D$, by 
the assumption $H^2$-pure we obtain that  both following maps $H^1(\bar X) \to H^1(D) \to H^1(\bar F_p)$ are surjective, hence their composition too. Then the global invariant cycle theorem \cite{De} yields that $\Mon \bar f =1$.

The last statement follows from the proof of Proposition \ref{p:equiv}.

\noindent
(b). and (c). follow by the same arguments of \cite{Di}, from the exact sequence (\ref{eq:exact1}) and the analysis given in the proof of Proposition \ref{p:trivmon}  using the condition that $D$ is a connected tree and that $\bar F_p$ is connected which imply that any horizontal divisor is totally ramified over each point $p\in \bar C\m C$.
  
 \end{proof}

From the statement and proof of Theorem \ref{t:monpoint} we get now easily the following, which recovers the results by Miyanishi-Sugie \cite{MS} and Dimca \cite{Di} cited in the Introduction.

\begin{corollary}
  Let $X$ be $H^2$ and $H^3$-pure. Let $C = \bC$ and let $\bar F_\infty$ be connected. 
  If $T_\ity =1$ and $b_1(\bar X)= 0$ then $f$ is rational and $\Mon f= 1$. 
 Reciprocally, if $f$ is rational and $T_\ity = 1$ then $\Mon f= 1$.
 \fin
 \end{corollary}

%

%


\section{Bounding the number of reducible fibres}\label{s:bound}

In this section we focus to another aspect, investigated by Kaliman \cite{Kal} in case of polynomials $f : \bC^2 \to \bC$, namely the number of reducible fibres. Kaliman proved that this number is at most $\delta -1$, where $\delta = \#$horizontal components of $D$, and remarked that this holds (with almost the same proof) for an acyclic surface $X$ instead of $\bC^2$. Our aim is to extend the setting to more general surfaces $X$. Let us remark that the number of reducible fibres is finite only if $f$ is primitive.

\begin{theorem}\label{t:one}
Let $f:X\to C$ be a primitive morphism from a $H^2$-pure and $H^3$-pure surface onto an affine curve. Then the number of reducible fibres of $f$ is at most $b_2(X)+ \delta -1$.
\end{theorem}

We need two lemmas. Let $\bar f : \bar X \to \bar C$ be a compact resolution of $f$. 
We use the following notations: $\bar F$ is the generic fibre of $\bar f$,  $\bar F_p := \bar f^{-1}(p)$ (and not the closure of the fibre $f^{-1}(p))$.

\begin{lemma}\label{eq:second}
Let $f:X\to C$ be a primitive morphism from a $H^2$-pure surface onto an affine curve $C$.
\begin{enumerate}
\item If $f$ is non-proper then 
$\sum_{p\in  C} b_2(F_p) \le  b_2(X)$.
\item If $f$ is proper then $\sum_{p\in  C} (b_2(F_p)-1) \le  b_2(X)-1$.
\end{enumerate}
\end{lemma}

\begin{proof} (a). 
Let $A_p$ denote the union of all the affine irreducible components of $F_p$. Then the intersection matrix of the components of the divisor $\sum_{p\in C}\cl(F_p \m A_p)$ (where $\cl$ means taking the closure in $\bar X$) is negative definite, by Zariski's Lemma, see \cite[Lemma 8.2, (9) and (10)]{BPV}. Therefore these components are linearly independent in $H_2(\bar X)$.   It follows that they  are 
independent in $H^2(X)$ since they have support in $X$ and since the inclusion $X\subset \bar X$ induces an inclusion in $H_2$ due to the assumed $H^2$-purity.
 The independence of these components yields now the desired inequality $\sum_{p\in  C} b_2(F_p) \le b_2(X)$. 
 
\noindent
(b).  Let 
$A_p$ be some irreducible component of $F_p$. Then the same proof as above yields that the components of the divisor $\sum_{p\in C}\cl(F_p \m A_p)$ are linearly independent in $H_2(X)$. Moreover, the general fibre $F$ is independent of these components  
\end{proof}

  Let us set the notations 
$k_p := b_1(\bar F) -b_1(\bar F_p)\ge 0$ (see the proof of Proposition \ref{p:mono}), $v_p := \#$vertical irreducible components of $\bar F_p$ in $D$, $a_p := \#$irreducible components of $\bar F_p$ which are not in $D$.

\begin{lemma}\label{l:first}
Let $f:X\to C$ be a non proper, primitive morphism from a $H^2$-pure surface onto an affine curve $C$. Then:
\begin{equation}\label{eq:first}
\sum_{p\in  \bar C} (a_p -1) + \sum_{p\in  \bar C} k_p  = \delta + 2 + b_2(X) -  b_1(X) - b_1(\bar X) - \chi(\bar F) \chi(\bar C).
\end{equation}
  
\end{lemma}

\begin{proof}
 One has the following relation between Euler characteristics, which holds for any regular map from a projective surface to a projective curve:
\begin{equation}\label{eq:chi}
 \chi (\bar X) - \chi (\bar F)\chi (\bar C) = \sum_{p\in \bar C} [ \chi (\bar F_p) - \chi (\bar F)].
\end{equation}

Using the primitivity assumption,  formula \eqref{eq:chi} yields:
\begin{equation}\label{eq:chi2} 
\begin{array}{c}
 2- 2b_1(\bar X) + b_2(\bar X) - \chi(\bar F) \chi(\bar C) = \\ 
 \sum_{p\in \bar C}[1-b_1(\bar F_p) + b_2(\bar F_p) -2 + b_1(\bar F)] 
=  \sum_{p\in \bar C}[ k_p + v_p + a_p - 1 ].
\end{array}
\end{equation}

Let us now consider the exact sequence of the pair $(\bar X, D)$:
\begin{equation}\label{eq:h} 
 0\to H^2(\bar X, D) \to H^2(\bar X) \to H^2(D) \to H^3(\bar X, D)\to H^3(\bar X) \to 0
\end{equation}
where the 0 at the left side is due to the $H^ 2$-purity of $X$ and the 0 at the right hand side comes from $H^3(D) = 0$. By Lefschetz duality (resp. Poincaré duality) we have the isomorphisms $H^2(\bar X, D) \simeq H_2(X)$, $H^3(\bar X, D) \simeq H_1(X)$, $H^2(\bar X) \simeq H_2(\bar X)$,  $H^3(\bar X) \simeq H_1(\bar X)$. We also have $b_2(D) = \delta + \sum_{p\in \bar C} v_p$, since $D$ is a tree with vertical and horizontal components.
We  obtain the following relation from \eqref{eq:h}:
\begin{equation}\label{eq:betti2}
 b_2(\bar X) = b_2(X) + \delta + \sum_{p\in \bar C} v_p - b_1(X) + b_1(\bar X)
\end{equation}
Substituting this expression of $b_2(\bar X)$ in \eqref{eq:chi2} yields \eqref{eq:first}.
\end{proof}

\subsection{Proof of Theorem \ref{t:one}}
The total number of reducible fibres of $f$ is at most $\sum_{p\in  C} (a_p -1)$.

 If $f$ is proper then $F_p$ is compact and  $a_p = b_2(F_p)$. Then the result follows directly from Lemma \ref{eq:second}(b), since $\delta =0$. 

Let us therefore assume in the following that $f$ is non-proper. The proof follows from a case-by-case study. 

\smallskip
\noindent
\textit{Case  $g(\bar C)>0$.}  By Theorem \ref{t:genus}, $f$  has a unique horizontal component, thus $\delta =1$, and is  simple. This implies that there is a single puncture at infinity for all fibres. Since the components of the fibres are separated by the resolution, it follows that any fibre of $f$ has a single affine irreducible component. We therefore have  $a_p -1 = b_2(F_p)$. Taking the sum over $p$ we get $\sum_{p\in  C} (a_p -1) = \sum b_2(F_p)\le b_2(X)$ from Lemma \ref{eq:second}(a). Our Theorem follows. 

\smallskip
\noindent
\textit{Case  $g(\bar C)=0$ and  $|\bar C  \m C| > 2$.} By  Theorem \ref{t:genus}, we get that $f$ is simple. 
Then as in the previous paragraph, we coclude the requered  inequality.

\smallskip
\noindent
\textit{Case 
$g(\bar C)= 0$ and $|\bar C \m C| = 2$. }
By definition, $a_p=0$ if $p\in \bar C \m C$. 
Under the hypotheses of $H^2$- and $H^3$-pure surface $X$ and $|\bar C \m C| = 2$, by the proof of Theorem \ref{c:points}(d), we  have $\delta=1$, $b_1(\bar F_p)=0$ for $p\in \bar C \m C$, $C\cong \bC ^*$  and $b_1(\bar X) = 0$. Moreover, $b_3(X)=0$ by Theorem \ref{c:points}(b). Moreover if ${\bar X}_C:={\bar f}^{-1} (C)$ denote the inverse image $C$, then we have $1=H^1(C)\subset H^1(\bar X_C) \subset H^1(X)$, hence $b_1(X) \ge 1$.

 Substituting these in \eqref{eq:first} we get :
\[ 
 \sum_{p\in  C} (a_p -1) + \sum_{p\in  C} k_p + b_1(X) -1 =  b_2(X).
\]
Now we have $k_p \ge 0$ for any $p\in C$. Hence the inequality: $\sum_{p\in  C} (a_p -1) \le b_2(X)$ which proves our theorem in this case.

\smallskip
\noindent
\textit{Case  $C \simeq \bC$.} Let us denote by $\ity$ the point $\bar C \m C$. Substituting in \eqref{eq:first} the terms $\chi (\bar C)=2$, $a_{\infty} =0$, $k_{\infty} = b_1(\bar F) - b_1(\bar F_{\infty})$, we get:
\begin{equation}\label{eq:new}
\sum_{p\in  C} (a_p-1 ) + 
\sum_{p\in  C} k_p  - b_1(\bar F) + b_1(X)  =   \delta - 1  + b_2(X)  + [b_1(\bar F_\ity)
 - b_1(\bar X)]  .
\end{equation}
Let us first show that the last couple of terms on the right hand side is cancelling. The following exact sequence is exact by using the $H^3$-purity (left) and the $H^2$-purity (right):
\[ 
0 \to H^1(\bar X, D) \to H^1(\bar X) \to H^1( D) \to 0,
\]
where $H^1(\bar X, D)\simeq H^3(X) =0$ by Theorem \ref{c:points}(b). But $H^1(D)\simeq H^1(\bar F_\infty)$ since $D$ is a tree and since $h^1 (D_i) =0$ for any the horizontal or vertical remaining divisor $D_i$, by the relation \eqref{eq:equality0}. 
Hence $b_1(\bar X) =b_1(\bar F_\ity)$.

Next, we show that $\sum_{p\in  C} k_p  - b_1(\bar F) + b_1(X)$ is non-negative.
Indeed, since $C \simeq \bC$, we may contract $\bC$ to a union of small disks around the set of critical values $A= \{ p_1, \ldots \}$ and union with simple non-intersecting paths connecting them with some exterior point in $\bC$. Then the pullback of this total union is homotopy equivalent to $\bar X_{\bC}$, and by using a Mayer-Vietoris argument, we have the following exact sequence:
\[ 
\oplus_i H_1(\bar F) \stackrel{\nu}{\to} H_1(\bar F)\oplus \oplus_i H_1(\bar F_{p_i}) \to H_1(\bar X_{\bC}) \to 0
\]
which shows that $\sum_{p\in C} k_p  - b_1(\bar F) + b_1(\bar X_{\bC}) \ge 0$. We may replace here $b_1(\bar X_{\bC})$ by $b_1(X)$ since $b_1(X) \ge b_1(\bar X_{\bC})$. Therefore from \eqref{eq:new} we finally get the inequality:
\[ \sum_{p\in  C} (a_p-1 )  \le   \delta - 1  + b_2(X)  
\]
which proves our claim. This ends the proof of our  theorem.
\fin

\begin{note}\label{n:1}
In the case $g(\bar C)=0$ one may also build a uniform proof starting from the approach used in the case $C \simeq \bC$.
\end{note}

\begin{note}\label{n:2}
We may actually prove a sharper statement. Let $X$ be a smooth surface and let $\bar X$ be some compactification of $X$ with $D := \bar X \m X$ a normal crossings divisor. Let $H^1(\Omega^1_{\bar X})$ be the cohomology group in $H^2(\bar X, \bC)$ generated by algebraic
cycles. Let $\Omega(X)$ denote the vector subspace of $H^2(X)$ generated by complete curves in $X$, that is $\Omega(X) := \im (H^1(\Omega^1_{\bar X}) \to H^1( \Omega^1_{\bar X}\langle D \rangle))\subset H^2(X)$, and let $\omega(X) := \dim \Omega(X)$.

Counting the complete curves in $X$ like done in  Lemma \ref{eq:second} (a) and (b), we obtain  $\sum_p b_2(F_p) \le \omega(X)$ in the non-proper case and $1+ \sum_p (b_2(F_p)- 1) \le \omega(X)$ in the proper case.  We conclude that one may replace $b_2(X)$ by $\omega(X)$ in Theorem \ref{t:one} and since by definition $\omega(X) \le b_2(X)$, we get a sharper inequality.

\end{note}


\section{Trivial mondromy with irreducible fibres}\label{localtriviality}

In this section we generalize the theorem of Dimca \cite{Di} that if a polynomial in $2$ variables  has irreducible 
fibers and  trivial monodromy, then it is an algebraically trivial fibration. The algebraicity occurs only because of the use of Abhyankar-Moh theorem which is possible only in case of $X= \bC^2$.

We therefore begin with some examples to explain what we
may expect in our more general context, namely when $X$ is a $H^2$ and $H^3$-pure surface. 

\begin{remark}\label{r:1}
 {\em Even if $f$ has trivial monodromy and irreducible fibres, one may not even get analytic local
triviality}.   Indeed, consider a smooth projective curve $\bar C$ of  genus $g(\bar C)>1$ and a point 
$p\in \bar C$. Let  ${\bar X}:=\bar C\times \bar C$ and let $D:=\Delta_{\bar C}\cup \{p\}\times \bar C$ be the union of the diagonal $\Delta_{\bar C}$ and a fibre.  Then
$X:=\bar X\m D$ is an $H^2$ and $H^3$-pure smooth surface. Let $C :=  \bar C \m \{p\}$. Then the restriction $f:X\to C$ of the first projection  has
trivial monodromy and irreducible fibres. It is not analytically locally trivial since  $\bar C\m \{x\}$ 
and $\bar C\m \{y\}$ cannot be isomorphic for all points $x$ and $y$ due to the fact that the set of automorphisms of a curve of genus $>1$ is finite. However $f$ is locally topologically  trivial.
\end{remark}

\begin{remark}\label{r:2} 
{\em Even if $f$ has  trivial monodromy and irreducible fibres, it can still have nonreduced fibres 
and hence may not be topologically locally trivial.}  To show this, consider the quadric hypersurface $Z$ in $\bC^3$ defined 
by the equation $x^2 - yz =0$. The $y$-coordinate  projection has 
every fibre isomorphic to $\bC$ but the fibre over $0$ is multiple,  defined by $x^2=0$ in the $(x,z)$ plane. This 
projection $Z\to \bC$ has a section  $\sigma: \bC\to Z$ given by $\sigma (y) = (0,y,0) $. Then $X:= Z\m \sigma(\bC)$ 
is a smooth $H^2$ and $H^3$-pure surface and the restriction of $f$ to $X$ has trivial monodromy since the fibres are 
affine lines. Each fibre is irreducible, isomorphic  $\bC^*$,  but there is a nonreduced fibre and hence cannot be 
locally trivial. 
\end{remark}

These remarks support the sharpness of the following theorem:

\begin{theorem}\label{t:fibration}\label{c:fibration}  
 Let $f:X\to C$ be a  morphism from a $H^2$ and $H^3$ pure surface onto an affine curve, 
with trivial monodromy group and such that all fibres of $f$ are  irreducible. Then:

\begin{enumerate}
\item All fibres of $f$ are diffeomorphic. 
\item If moreover all fibres are assumed to be reduced, then $f$ is a locally trivial fibration, i.e. the bifurcation set of $f$ is empty. 
\item If $f$ has at least one nonreduced fibre, 
then the general fibre of $f$ is isomorphic either to $\bC$ or  to $\bC^*$, or to an elliptic curve. 

\end{enumerate}

\end{theorem}

\begin{proof} 
(a). Consider $\bar{f}: {\bar X} \to {\bar C}$, a compact resolution of $f$ with a smooth projective completion 
$\bar X$ of X and  $D:= {\bar X}\m X$, a divisor with simple normal crossings. For $p\in {\bar C}\m C$, we have ${\bar F}_p\subset D$. Since $f$ has trivial mondromy,  $f$ is simple and primitive by Proposition \ref{p:trivmon}. Also  Proposition~\ref{p:mono} implies that $H^1({\bar F}_t)\cong H^1({\bar F}_p)$ 
 for all $t\in {\bar C}$, where ${\bar F}_t$ is the  fibre of $\bar{f}$. Let $F_t$ be the 
fibre of $f$ and $\cl(F_t)$ denote the closure of $F_t$ in $\bar X$.  Since $F_t$ is irreducible, every other components of 
${\bar F}_t$ are contained in $D$. Moreover $D$ being a tree by the $H^2$-purity, the number of points in $\cl(F_t)\m F_t$ is the same. Indeed, any reduction
 in this number would produce the intersection of two horizontal components 
which creates a loop as they alredy intersect the fibre $F_p$. On the other hand, an increase in this number would imply that $D$ is not connected. By \eqref{eq:inequality} each vertical component of $D$ that is not contained in 
${\bar F}_p$ is a rational curve.  Thus we get the isomorphisms 
$ H^1({\bar F}_t) \cong H^1(\cl(F_t)) \cong H^1({\bar F}_p)$ for all $t\in C$. 
This shows that all $F_t$ have the same topology and in particular that all have nonsingular reduced structure.  Since open Riemann surfaces, they are also diffeomorphic.\\

\noindent 
(b). Let $D_t$ be the union of vertical components of $D$ lying over $t\in C$. Every connected 
component of $D_t$  intersects $\cl(F_t)$ in a unique point and 
 distinct connected components of $D_t$ intersect $\cl(F_t)$ in distinct points. Each connected 
component of  $D_t$ is negative definite and hence blows down to a normal analytic singularity by Grauert \cite[Satz 3.1]{Gr}.
After blowing down all connected components of $D_t$, for points $t\in  C$, we obtain a normal compact analytic surface $X'$ together with 
the induced map $f': X'\to {\bar C}$. Now for each $t\in C$, the arithmetic genus  $p_a(\cl(F_t))$  of the closure $\cl(F_t)$ is equal to the arithmetic 
genus  $p_a({\bar F})$ of the general fibre of $\bar f$. This follows as the monodromy is trivial and all other components of ${\bar F_t}$ are contained 
in $D$ which is a  tree of smooth rational curve. Let ${\bar F_t}'$  denote the fibre of $f'$ over $t\in C$ and  $\bar F'$ denote the general fibre, 
respectively. Then the  map  $\cl(F_t)\to {\bar F_t}'$ is bijective and hence we have:
$$ p_a ({\bar F'}) = p_a({\bar F})  = p_a(\cl(F_t)) \leq p_a(({\bar F'}_t)_{\rm red})$$ 
We also have
$$ p_a({\bar F'}_t) = p_a({\bar F'}) = p_a({\bar F}). $$

If ${\bar F_t}'$ is reduced, then all these will be equal and  it follows that ${\bar F_t}'$ is smooth, for all $t\in C$.
Since $C$ is also smooth, it implies that $X'$ is smooth. All fibres have the same number of points at infinity and these are smooth points, we
conclude that $X\to C$ is a locally trivial fibration by  the relative Ehresmann Fibration Theorem.

To show now that ${\bar F_t}'$ is reduced, we note that all fibres of $f'$ are generically reduced as fibres of $f$ are 
reduced and $f$ and $f'$ coincide on $X$. So we only need to check that the fibre ${\bar F_t}'$ has no embedded component. 
On the other hand $X'$ is a normal surface and hence is Cohen Macaulay. On a Cohen Macaulay germ, any Cartier divisor which 
is generically reduced is a reduced divisor. Hence ${\bar F_t}'$ is reduced proving the Theorem.\\

\noindent
(c).
By (a), all fibres of $f$ are smooth (when considered with reduced structure)  and diffemorphic. Let $t_1, t_2, \ldots, t_r\in C$ be the points where $f$ has 
multiple fibres with multiplicities $m_1,m_2, \ldots, m_r>1$ respectively. Let $g:C'\to C$ be a ramified Galois cover of $C$ of ramification index $m_i$ 
at $t_i\in C$ for $i=1,2, \ldots, r$, which means that at any ramification point $t'_i\in C'$ over $t_i$ the index is the same, equal to $m_i$. One can  always find such a cover, eventually with some more branch points in $C$.   Let $h:Y\to C'$ and $h':Y'\to C'$ be the normalization of the  fibre product of $f,g$ and $f',g$ respectively.  This construction implies that all fibres of $h:Y\to C'$ are  smooth reduced 
and the restriction map $h^{-1}(t'_i) \to f^{-1}(t_i)_{\rm red}$ is \'etale, which fact will be used later in this proof. 
 


\smallskip
\noindent
\textbf{Loops in components of fibres of $h':Y'\to C'$.}  Note that  $h'^{-1}(t')\cong f'^{-1} (g(t'))$ for all $t'\in C'$ 
such that $g(t')\neq t_i$.  The only structural 
change happens at fibres  $h'^{-1}(t'_i)$ for  $t'_i \in g^{-1}(t_i)$. For any $t'_i\in g^{-1}(t_i)$,  $h'^{-1}(t'_i) \to f'^{-1} (t_i)$  is a ramified cover of degree $m_i$, which is \'etale 
over $f^{-1}(t_i)$.  Let $D_1, \cdots, D_k$ be the horizontal components of $X'\to \bar C$.  By the definition of $Y$ and of $Y'$ the inverse image $D'_i$ of $D_i$ in $Y'$  is a curve isomorphic to $C'$ which forms a section of $h'$. Hence every component of the fibre ${h'}^{-1}(t'_i)$
has $k$ points at infinity (i.e. points that are not in the fibre of $h$) and they are the same for any $i$. Then the projection ${h'}^{-1}(t'_i) \to f'^{-1}(t_i)$ is totally 
ramified at these points. Consequently, if $k>1$ and if ${h'}^{-1}(t'_i)$  is reducible then there are loops in the dual graph of this fibre. 

\smallskip
\noindent
\textbf{ Genus of the components of fibres of $h':Y'\to C'$. }
Without loss of generality let us choose  $i=1$ and  fix a point $t'_1$ lying over  $t_1\in C$. Let ${h'} ^{-1}(t'_1):= C'_1\cup\cdots \cup C'_r$ be 
the decomposition of the fibre ${h'} ^{-1}(t'_1)$ into its irreducible components.  Then $p_a(C_j)\geq p_a({f'}^{-1}(t_1))$ for all $j$, and moreover  $f'$ 
and $h'$ have isomorphic general fibres.  Let us denote by $t'\in C'$ be a general point (near $t'_1$) and let  $t:=g(t')\in C$.  Then we get:

$$  p_a({h'}^{-1}(t'))  =  p_a({h'} ^{-1}(t'_1)) \geq  \sum_{j=1}^r  p_a(C_j) \geq r p_a( {f'}^{-1}(t_1)) =r p_a({f'}^{-1}(t)) = rp_a({h'}^{-1}(t)) $$

The first inequality follows from the fact that the  fibre  ${h'} ^{-1}(t'_1)$ is generically reduced and hence is reduced as $Y'$ is Cohen Macaulay, and the fact that the arithmetic genus of a reduced curve is at least the sum of the genera of each component. Now we analyse the 
following cases.

{\underline {\it Case 1:}} $p_a({h'}^{-1}(t'))>1$. From the above sequence it follows that if $p_a({h'}^{-1}(t')) >0$ then $r=1$ and therefore all the inequalities are equalities.  But if $p_a({h'}^{-1}(t)) >1$, the Riemann-Hurwitz theorem implies $p_a(C_1) >  p_a( f'^{-1}(t_1))$,  contradiction. 

{\underline{\it Case 2:}} $p_a({h'}^{-1}(t'))=1$. Then again $r=1$ and all inequalities are equalities, thus  $p_a({f'}^{-1}(t)) =  p_a(C_1)= 1$. By the first part of Theorem~\ref{t:fibration}, the  special 
fibre of $f'$ has the same genus, i.e.,   $p_a(f'^{-1}(t_1)) = 1$. 
Since  $C_1 \to f'^{-1}(t_i)$  is a degree $m_i>1$ covering of curves of genus 1 with total ramifications at the points at infinity,   the Riemann-Hurwitz argument yields again a contradiction.

 We therefore conclude that $p_a({h'}^{-1}(t))$ must be equal to $0$, 
hence all the curves in the above displayed sequence are rational. 

{\underline{ \it Case 3:}} $p_a({h'}^{-1}(t'))=0$.  This implies that the fibre ${h'}^{-1}(t'_1)$ cannot have ``loops'' in its dual graph. If $k>1$, the above argument of ``loops in components'' shows that this fibre is irreducible. Thus
the fibres  ${h}^{-1}(t'_1)$ and  ${f}^{-1}(t')$ are a $\bP^1$ minus $k$ points and the map between them is \'etale, which implies that $k=2$, i.e. exactly two sections at infinity. It follows that the general fibre of $f$ must be $\bC^*$.

If $k=1$ then the general fibre of $f$ is $\bC$, and each fibre $h^{-1}(t'_i)$ consists of $m_i$ disjoint copies of $f^{-1}(t_i)\cong \bC$, as there are no non-trivial 
\'etale covers of $\bC$. 
\end{proof}

\begin{remark}
From the last part of the above proof,  under the same hypotheses, we moreover get the following byproduct:

If the general fibre of $f$ is  $\bC^*$, then all the fibres of $h'$ are
irreducible and, as in the proof of Theorem~\ref{t:fibration}, $h'$  is a locally trivial fibration.  Since every fibre of $h'$ is a smooth $\bP^1$,
it follows that $Y'$ is  smooth and $Y'\to X'$ is an \'etale cover of $X'\m \cup_{i,j} \{ {f'}^{-1}(t_i)\cap D_j\}$. Hence $X'$ is the quotient of a 
smooth surface by a finite group (in fact the Galois group of $C'\to C$).

If the general fibre of $f$ is  $\bC$, then $Y'\to X'$ is an \'etale cover outside $\cup_{i} \{ {f'}^{-1}(t_i)\cap D_1 \}$, but $h'$ cannot be a fibration and actually $Y'$ cannot be a non-singular variety.
\end{remark}



\begin{thebibliography}{MMM}

\bibitem[ACD]{ACD}
E. Artal-Bartolo, P. Cassou-Nogu\`es,  A. Dimca,  {\em Sur la topologie des polyn\^omes complexes}.  Singularities (Oberwolfach, 1996),  317--343, Progr. Math., 162, Birkh\"auser, Basel, 1998.

\bibitem[BPV]{BPV}
 W. Barth, C. Peters, A.  Van de Ven, Compact complex surfaces. Ergebnisse der Mathematik und ihrer Grenzgebiete (3), 4. Springer-Verlag, Berlin, 1984.

\bibitem[Cl]{Cl}
C.H. Clemens, \textit{Degeneration of K\" ahler manifolds}.
Duke Math. J. 44 (1977), no. 2, 215--290. 

\bibitem[De1]{De}
P. Deligne, {\em Th\'eorie de Hodge. II}. 
Inst. Hautes \' Etudes Sci. Publ. Math. No. 40 (1971), 5--57.

\bibitem[De2]{De3}
P. Deligne, {\em Th\'eorie de Hodge. III}. 
Inst. Hautes \' Etudes Sci. Publ. Math.  No. 44 (1974), 5--77. 

\bibitem[Di]{Di}
A. Dimca,  {\em Monodromy at infinity for polynomials in two variables}.  J. Algebraic Geom.  7  (1998),  no. 4, 771--779.

\bibitem[Gr]{Gr} H. Grauert, 
\emph{\"Uber Modifikationen und exzeptionelle analytische Mengen}. 
Math. Ann. 146 (1962), 331--368.

\bibitem[GS1]{GS1}
 R.V. Gurjar, A.R. Shastri,  \emph{On the rationality of complex homology 2-cells. I.} J. Math. Soc. Japan 41 (1989), no. 1, 37--56. 
\bibitem[GS2]{GS2}
 R.V. Gurjar, A.R. Shastri,  \emph{On the rationality of complex homology 2-cells. II.} J. Math. Soc. Japan 41 (1989), no. 2, 175--212.

\bibitem[GPS]{GPS}
R.V. Gurjar, C.R. Pradeep, A.R. Shastri, \emph{On rationality of logarithmic Q-homology planes. II.} Osaka J. Math. 34 (1997), no. 3, 725--743. 

\bibitem[GP]{GP}
R.V. Gurjar, C.R. Pradeep,  {\em $Q$-homology planes are rational. III}.  Osaka J. Math.  36  (1999),  no. 2, 259--335.

\bibitem[Lo]{Lo}
D. Lorenzini, {\em
Reducibility of polynomials in two variables.} J. Algebra 156 (1993), no. 1, 65-75. 

\bibitem[Kal]{Kal}
S. Kaliman,  {\em Two remarks on polynomials in two variables}.  Pacific J. Math.  154  (1992),  no. 2, 285--295.

\bibitem[MS]{MS}
M. Miyanishi, T. Sugie, {\em Generically rational polynomials}.
Osaka J. Math. 17 (1980), no. 2, 339--362. 

\bibitem[NN1]{NN1}
W. Neumann, P. Norbury,  {\em Vanishing cycles and monodromy of complex polynomials}.  Duke Math. J.  101  (2000),  no. 3, 487--497. 

\bibitem[NN2]{NN2}
W. Neumann, P. Norbury,  {\em Rational polynomials of simple type}.  Pacific J. Math.  204  (2002),  no. 1, 177--207.

\bibitem[ST]{ST}
D. Siersma, M. Tib\u ar,  {\em Singularities at infinity and their vanishing cycles. II. Monodromy}.  Publ. Res. Inst. Math. Sci.  36  (2000),  no. 6, 659--679.

\bibitem[St]{St}
Y. Stein, {\em The total reducibility order of a polynomial in two variables}. Israel J. Math. 68 (1989), no. 1, 109-122.

\bibitem[Vi]{Vi} A. Vistoli, 
{\em The number of reducible hypersurfaces in a pencil.}
Invent. Math. 112 (1993), no. 2, 247-262. 

\bibitem[Vo]{Vo}
C. Voisin,  {\em Hodge theory and complex algebraic geometry. II}. Reprint of the 2003 English edition. Cambridge Studies in Advanced Mathematics, 77. Cambridge University Press, Cambridge, 2007.

\bibitem[Za1]{Za1}
M.G. Zaidenberg, 
\textit{Isotrivial families of curves on affine surfaces, and the characterization of the affine plane}. Izv. Akad. Nauk SSSR Ser. Mat. 51 (1987), no. 3, 534--567; translation in
Math. USSR-Izv. 30 (1988), no. 3, 503--532 

\bibitem[Za2]{Za2}
M.G. Zaidenberg, 
\textit{Additions and corrections to the paper: ``Isotrivial families of curves on affine surfaces, and the characterization of the affine plane''}.  Izv. Akad. Nauk SSSR Ser. Mat. 55 (1991), no. 2, 444--446; translation in
Math. USSR-Izv. 38 (1992), no. 2, 435--437. 


\end{thebibliography}
\end{document}